\documentclass[reqno]{amsart}

\usepackage{amsmath,amsthm,amssymb,amsfonts,mathtools}
\usepackage{microtype}
\usepackage{enumitem}
\usepackage[hidelinks]{hyperref}
\hypersetup{
  pdftitle={Algebraic defect and positive mother-body measures},
  pdfauthor={Boris Shapiro},
  pdfkeywords={Algebraic Cauchy transforms; mother-body measures; algebraic defect; Herglotz functions; quadratic differentials}
}

\newtheorem{theorem}{Theorem}[section]
\newtheorem{proposition}[theorem]{Proposition}
\newtheorem{lemma}[theorem]{Lemma}
\newtheorem{corollary}[theorem]{Corollary}

\newtheorem{question}[theorem]{Question}

\newtheorem{theoremalph}{Theorem}

\theoremstyle{definition}
\newtheorem{definition}[theorem]{Definition}
\newtheorem{example}[theorem]{Example}
\theoremstyle{remark}
\newtheorem{remark}[theorem]{Remark}

\newcommand{\bC}{\mathbb C}
\newcommand{\bR}{\mathbb R}

\newcommand{\M}{\mathcal M}
\newcommand{\Cauchy}{\mathcal C}
\newcommand{\supp}{\operatorname{supp}}
\newcommand{\Conv}{\operatorname{Conv}}
\newcommand{\dist}{\operatorname{dist}}

\newcommand{\TV}{\mathrm{TV}}
\newcommand{\eps}{\varepsilon}
\newcommand{\Hone}{\mathcal H^1}
\newcommand{\Db}{\mathbb D}
\renewcommand{\Re}{\operatorname{Re}}
\renewcommand{\Im}{\operatorname{Im}}

\title[Algebraic defect and mother-body measures]{Algebraic defect and positive mother-body measures}
\author{Boris Shapiro}
\address{Department of Mathematics, Stockholm University, SE-106 91 Stockholm, Sweden}
\email{shapiro@math.su.se}
\date{}

\begin{document}

\begin{abstract}
Continuing the study of mother-body measures with algebraic Cauchy transform,
we associate with a positive algebraic germ $f$, its irreducible equation
$P(z,w)=0$, and a compact convex set $K$ the functional
\[
 \mathfrak D_{P,K}(\mu)=
 \int_K\left|P\bigl(z,\Cauchy_\mu(z)\bigr)\right|^{1/d}\,dA(z),
 \qquad d=\deg_wP,
\]
on the set of positive measures supported in $K$ and whose Cauchy transform has germ $f$ at infinity. We prove
continuity of $\mathfrak D_{P,K}(\mu)$ and attainment of its minimum, characterize zero defect, and obtain
the estimate
\[
 \mu(D(a,r))\le C_1r+C_2\mathfrak D_{P,K}(\mu)/r
\]
away from the zero set of the leading coefficient of $P$, where $D(a,r)$ is the disk of radius $r$ centered at $a$.  We establish a
 dual formula for the defect in the rational case and construct positive algebraic Cauchy
transforms by Herglotz theory, Fuss--Catalan and Raney laws, positive sums,
polynomial pushforwards, and branch graphs.  The passage from zero defect to a
mother body is made under the planar-null hypothesis of the  support.
\end{abstract}



\maketitle

\section{Introduction}

This paper continues the study initiated in~\cite{BoSh}. 
We introduce a variational obstruction to the existence of a positive
mother-body measure and construct several new classes of algebraic Cauchy
transforms admitting such measures.

Let
\[
 f(z)=\frac{a_0}{z}+O(z^{-2}),\qquad a_0>0,
\]
be an algebraic germ at infinity, and let $P\in\bC[z,w]$ be its irreducible
defining polynomial.  Following~\cite{BoSh}, $f$ is called \emph{positive}.  Every
positive analytic germ is the exterior Cauchy transform of a positive compactly
supported measure~\cite[Theorem~1]{BoSh}.  A positive \emph{mother-body measure}
for $f$ is such a measure whose support is a finite union of compact
semi-analytic arcs and isolated points and whose Cauchy transform is a branch of
$P(z,w)=0$ in every component of the complement of the support.

Throughout, $dA$ denotes planar Lebesgue measure and $A(E)$ the planar
Lebesgue measure of a measurable set $E$. 
Fix a compact convex set $K$ and let $\M_f(K)$ be the class of positive measures
supported in $K$ whose Cauchy transform has germ $f$ at infinity.  Put
\[
 \mathfrak D_{P,K}(\mu)=
 \int_K\left|P\bigl(z,\Cauchy_\mu(z)\bigr)\right|^{1/d}\,dA(z),
 \qquad d=\deg_wP.
\]
The exponent $1/d$ gives linear growth in $|\Cauchy_\mu|$ away from the zero set
of the leading coefficient of $P$.

Our main results are as follows. 

\begin{theoremalph}\label{thmA}
Assume $\M_f(K)\ne\varnothing$.  Then $\mathfrak D_{P,K}$ is finite and
weak-$*$ continuous on $\M_f(K)$ and attains its minimum
$\delta_P(f;K)$.  Moreover, $\delta_P(f;K)=0$ if and only if there exists
$\mu\in\M_f(K)$ such that
\[
 P\bigl(z,\Cauchy_\mu(z)\bigr)=0
 \qquad\text{for almost every }z\in\bC.
\]
Every minimizer has this property when the minimum is zero.
\end{theoremalph}


Write $p_d(z)$ for the leading coefficient of $P$ in $w$, and for $a\in\bC$
and $r>0$ set
\[
 D(a,r):=\{z\in\bC:|z-a|<r\}.
\]

\begin{theoremalph}\label{thmB}
Let $L\subset\bC\setminus p_d^{-1}(0)$ be compact.  There exist
$C_1,C_2,r_0>0$, depending only on $P$, $L$, and $K$, such that
\begin{equation}\label{eq:intro-concentration}
 \mu(D(a,r))\le C_1r+\frac{C_2}{r}\mathfrak D_{P,K}(\mu)
\end{equation}
for $\mu\in\M_f(K)$, $a\in L$, and $0<r\le r_0$.  Consequently, a zero-defect
measure has linear growth on $L$ and no atoms there.  If $\mu$ has an atom of
mass $m$ at a point of $L$, then
\[
 \mathfrak D_{P,K}(\mu)\ge c_{\rm at}m^2
\]
for a constant $c_{\rm at}>0$ depending only on $P$, $L$, $K$, and $a_0$.
\end{theoremalph}

The paper is organized as follows.  Section~\ref{sec:BoSh-background}
recalls the two existence theorems from~\cite{BoSh} used below and explains
their relation to the defect functional.  Section~\ref{sec:moment} introduces
the exterior moment classes and records the required estimates for Cauchy
transforms.  Section~\ref{sec:defect} proves finiteness, weak-$*$ continuity,
and attainment of the algebraic defect, together with the zero-defect
criterion.  Section~\ref{sec:quantitative} establishes the concentration
estimate and its consequences for atoms and linear growth.  Section~\ref{sec:support} relates zero defect to mother-body measures under
the planar-null support hypothesis.  Section~\ref{sec:rational} gives an exact $L^1$ dual formula and a
complete description of zero defect for rational germs.  Section~\ref{sec:constructive}
develops constructive classes, including the Herglotz criterion, the
Fuss--Catalan and Raney families, positive superposition, polynomial
pushforwards, and the branch-graph criterion.  Section~\ref{sec:quadratic}
records the resulting connection with the quadratic-differential theory
of~\cite{BoSh}.  Section~\ref{sec:stability} proves lower semicontinuity and a
recovery-sequence criterion for continuity of the minimum defect.  Section~\ref{sec:open} concludes the paper with open problems concerning
support regularity, effective existence, and quantitative lower bounds.

\section{Known existence results}\label{sec:BoSh-background}

We recall two results from~\cite{BoSh}.  This paper writes the algebraic equation
as $P(\Cauchy,z)=0$, with the Cauchy-transform variable first; below we retain the
present convention $P(z,w)=0$.

The first result concerns only the prescribed germ at infinity.  Its proof gives
the following form of~\cite[Theorem~1]{BoSh}. 

For a finite complex Borel measure
$\mu$ with compact support in $\bC$, let $\Cauchy_\mu(z)$ denote its Cauchy transform, i.e., 
\[
 \Cauchy_\mu(z)=\int_{\bC}\frac{d\mu(\zeta)}{z-\zeta}.
\]

\begin{theorem}[Exterior realization]
\label{thm:BoSh-exterior}
Let
\[
 f(z)=\frac{a_0}{z}+O(z^{-2}),\qquad a_0>0,
\]
be analytic in a neighborhood of infinity.  There exist a real-analytic Jordan
curve $\Gamma\subset\bC$ and a positive measure $\sigma$ of mass $a_0$ supported
on $\Gamma$ such that its Cauchy transform $ \Cauchy_\sigma(z)$ satisfies 
\[
 \Cauchy_\sigma(z)=f(z)
\]
in the unbounded component of $\bC\setminus\Gamma$.  In particular,
$\Cauchy_\sigma$ has germ $f$ at infinity.
\end{theorem}

The measure in Theorem~\ref{thm:BoSh-exterior} is obtained from the equilibrium
measure on a sufficiently large regular level curve of a harmonic primitive of
$f$.  The theorem is an exterior realization statement: it does not assert that
$\Cauchy_\sigma$ follows branches of a fixed algebraic equation in the bounded
component of $\bC\setminus\Gamma$, and hence does not by itself produce a
mother-body measure.

\medskip
We next recall the balanced condition of a bivariate polynomial $P$.  Write
\[
 P(z,w)=\sum_{(j,k)\in S(P)}a_{j,k}z^jw^k,
 \qquad a_{j,k}\ne0,
\]
and put
\[
 m(P)=\min_{(j,k)\in S(P)}(k-j).
\]
Polynomial $P$ is \emph{balanced} if $m(P)=0$.  A diagonal monomial $z^nw^n$
is called \emph{dominant} if
\[
 j\le n\quad\text{and}\quad k\le n
 \qquad\text{for every }(j,k)\in S(P).
\]
Following~\cite{BoSh}, an irreducible balanced polynomial $P$ is called
\emph{excellent balanced} if it has a unique probability branch
$f(z)=z^{-1}+O(z^{-2})$, its support $S(P)$ contains a dominant diagonal
monomial, and this probability branch is the only positive branch of $P(z,w)=0$.
The phrase ``lexicographically bigger'' in~\cite{BoSh} means precisely the
coordinatewise dominance above. \cite[Theorem~5]{BoSh} claims the following. 

\begin{theorem}[Excellent balanced equations] 
\label{thm:BoSh-excellent}
Let $P\in\bC[z,w]$ be excellent balanced, and let
$f(z)=z^{-1}+O(z^{-2})$ be its probability branch.  Then $f$ admits a positive
probability mother-body measure.
\end{theorem}

  The following consequences explain how these two
results enter the present variational framework.

\begin{proposition}\label{prop:BoSh-defect-connection}
\begin{enumerate}[label=\textup{(\roman*)}]
\item For every positive algebraic germ $f$ there is a compact convex set $K_0$
such that $\M_f(K)\ne\varnothing$ for every compact convex $K\supset K_0$.
Consequently the minimum defect $\delta_P(f;K)$ is defined and attained for all
such $K$.
\item If $P$ is excellent balanced and $f$ is its probability branch, then there
is a compact convex set $K_0$ such that
\[
 \delta_P(f;K)=0
\]
for every compact convex $K\supset K_0$.
\end{enumerate}
\end{proposition}

\begin{proof}
For \textup{(i)}, apply Theorem~\ref{thm:BoSh-exterior} and take
$K_0=\Conv(\Gamma)$.  The resulting measure belongs to $\M_f(K)$ whenever
$K\supset K_0$; attainment then follows from Theorem~\ref{thmA}.  For
\textup{(ii)}, let $\mu$ be the positive mother-body measure supplied by
Theorem~\ref{thm:BoSh-excellent} and take
$K_0=\Conv(\supp\mu)$.  Since
$P(z,\Cauchy_\mu(z))=0$ almost everywhere, $\mathfrak D_{P,K}(\mu)=0$ for every
$K\supset K_0$.
\end{proof}

Thus Theorem~\ref{thm:BoSh-exterior} supplies feasibility, whereas
Theorem~\ref{thm:BoSh-excellent} supplies zero defect.  The distinction is
essential: an exterior representing measure need not be a mother body.  The
constructive results in Section~\ref{sec:constructive} extend the second theorem
beyond the excellent-balanced class; in particular, the Fuss--Catalan family is
balanced but not excellent balanced and nevertheless has zero defect.

\section{Moment classes and algebraic germs}\label{sec:moment}

Recall that for any compactly supported measure $\mu$, its Cauchy transform $\Cauchy_\mu(z)$ converges absolutely for almost every $z$ and defines an element of
$L^1_{\mathrm{loc}}(\bC)$, holomorphic off $\supp\mu$.  Distributionally,
\begin{equation}\label{eq:dbar-Cauchy}
 \partial_{\bar z}\Cauchy_\mu=\pi\mu ,
\end{equation}
and $\Cauchy_\mu$ is the unique $L^1_{\mathrm{loc}}$ solution of
\eqref{eq:dbar-Cauchy} tending to $0$ at infinity.  Near infinity,
\[
 \Cauchy_\mu(z)=\sum_{n\geq 0}\frac{m_n(\mu)}{z^{n+1}},
 \qquad
 m_n(\mu)=\int_{\bC}\zeta^n\,d\mu(\zeta).
\]

Let $f$ be a positive algebraic germ at infinity,
\[
 f(z)=\frac{a_0}{z}+O(z^{-2}),\qquad a_0>0,
\]
and choose its irreducible defining polynomial
\begin{equation}\label{eq:minimal-polynomial}
 P(z,w)=\sum_{j=0}^{d}p_j(z)w^j\in\bC[z,w],
 \qquad d=\deg_wP\geq 1,\quad p_d\not\equiv0,
\end{equation}
so that $P(z,f(z))=0$ near infinity.  The polynomial $P$ is unique up to a
nonzero constant factor.  Set 
\begin{equation}\label{eq:polar-set}
 \Pi_P:=p_d^{-1}(0)\subset\bC.
\end{equation}
Note that  off $\Pi_P$ the equation $P(z,w)=0$ has exactly $d$ roots counted
with multiplicity, and these are locally bounded (Lemma~\ref{lem:reverse}
below).  The set $\Pi_P$ contains the poles of every branch of the algebraic
function.

Write
\[
 f(z)=\sum_{n\ge0}\frac{m_n(f)}{z^{n+1}}
\]
near infinity.  Suppose first that a positive compactly supported representative
$\sigma$ of the germ $f$ is given, i.e. $\Cauchy_\sigma(z)=f$ near $\infty$ and set $K=\Conv(\supp\sigma)$.  More generally,
for any compact convex set $K\subset\bC$ define the exterior moment class
\begin{equation}\label{eq:moment-class}
 \M_f(K)=\left\{\mu\ge0:\ \supp\mu\subset K,\quad
 m_n(\mu)=m_n(f)\ \text{for every }n\ge0\right\}.
\end{equation}
Equivalently, $\Cauchy_\mu$ has the germ $f$ at infinity.  Since
$\bC\setminus K$ is connected, all measures in $\M_f(K)$ have the same Cauchy
transform on $\bC\setminus K$; we continue to denote this common holomorphic
function by $f$.  Moreover, $P(z,\Cauchy_\mu(z))$ vanishes there by the identity
theorem, because it vanishes near infinity.  If
$K=\Conv(\supp\sigma)$, then $\sigma\in\M_f(K)$, so the class is nonempty.

Fix a bounded open set $\Omega$ with $K\subset\Omega$.
 For a finite complex Borel measure $\mu$ on $K$, we write $|\mu|$ for its
\emph{total-variation measure} and
\begin{equation}\label{eq:TV-definition}
 \|\mu\|_{\TV}:=|\mu|(K).
\end{equation}
Equivalently,
\[
 \|\mu\|_{\TV}
 =\sup\left\{\left|\int_K\varphi\,d\mu\right|:
               \varphi\in C(K),\ \|\varphi\|_\infty\le1\right\}.
\]
Thus $\|\mu\|_{\TV}$ measures the total size of a complex measure, including
possible cancellation.  In the applications below the measures are usually
positive, in which case $|\mu|=\mu$ and $\|\mu\|_{\TV}=\mu(K)$ is simply the
total mass.

\begin{lemma}[Compactness and convexity of the constrained class]
\label{lem:compact-class}
$\M_f(K)$ is convex and weak-$*$ compact in the space $\M(K)=C(K)^*$ of finite
measures on $K$.  More generally, for every $M>0$ the set
\[
 \mathcal B_M(K)=\{\mu\in\M(K):\|\mu\|_{\TV}\leq M\}
\]
is weak-$*$ compact and weak-$*$ metrizable, so that weak-$*$ continuity on
$\mathcal B_M(K)$ may be tested on sequences.
\end{lemma}

\begin{proof}
Convexity of $\M_f(K)$ is immediate from the linearity of $\mu\mapsto m_n(\mu)$
and of the positivity constraint.  Every $\mu\in\M_f(K)$ is positive of mass
$m_0(\mu)=a_0$, hence lies in $\mathcal B_{a_0}(K)$.  The set
$\mathcal B_M(K)$ is the closed ball of $C(K)^*$, weak-$*$ compact by
Banach--Alaoglu and weak-$*$ metrizable because $C(K)$ is separable.  Positivity
and each moment condition $\mu\mapsto\int_Kz^n\,d\mu$ are weak-$*$ closed
conditions, so $\M_f(K)$ is a weak-$*$ closed subset of $\mathcal B_{a_0}(K)$.
\end{proof}

We use the following $L^q$ estimate.

\begin{lemma}[$L^q$ bound, $q<2$]\label{lem:Lq}
Let $1\leq q<2$ and let $R=\sup\{|z-\zeta|:z\in\Omega,\ \zeta\in K\}$.  Then for
  every finite complex measure $\mu$ supported in $K$,
\begin{equation}\label{eq:Lq-bound}
 \|\Cauchy_\mu\|_{L^q(\Omega)}
 \leq A_q\|\mu\|_{\TV},
 \qquad
 A_q=\Bigl(\frac{2\pi R^{2-q}}{2-q}\Bigr)^{1/q}.
\end{equation}
\end{lemma}

\begin{proof}
If $\|\mu\|_{\TV}=0$, then $\mu=0$ and the assertion is immediate.  Assume
therefore that $\|\mu\|_{\TV}>0$.  The kernels $|z-\zeta|^{-1}$ and $|z-\zeta|^{-q}$ are locally
integrable in the plane because $q<2$.  Fubini's theorem shows, in
particular, that
\[
 \int_K\frac{d|\mu|(\zeta)}{|z-\zeta|}<\infty
\]
for almost every $z\in\Omega$, so the following pointwise estimates are
legitimate at almost every $z$.  The triangle inequality for integration with
respect to a complex measure gives
\[
 |\Cauchy_\mu(z)|
 \leq\int_K\frac{d|\mu|(\zeta)}{|z-\zeta|}.
\]
Apply Jensen's inequality to the probability measure
$d\nu=d|\mu|/\|\mu\|_{\TV}$ and the convex function $t\mapsto t^q$.  We obtain
\begin{align*}
 |\Cauchy_\mu(z)|^q
 &\leq\left(\int_K\frac{d|\mu|(\zeta)}{|z-\zeta|}\right)^q\\
 &=\|\mu\|_{\TV}^{\,q}
   \left(\int_K\frac{d\nu(\zeta)}{|z-\zeta|}\right)^q\\
 &\leq\|\mu\|_{\TV}^{\,q}
   \int_K\frac{d\nu(\zeta)}{|z-\zeta|^q}
 =\|\mu\|_{\TV}^{\,q-1}
   \int_K\frac{d|\mu|(\zeta)}{|z-\zeta|^q}.
\end{align*}
Integrating over $\Omega$ and using Tonelli's theorem for the nonnegative
integrand gives
\begin{align*}
 \int_\Omega|\Cauchy_\mu(z)|^q\,dA(z)
 &\leq\|\mu\|_{\TV}^{\,q-1}\int_K
 \left(\int_\Omega\frac{dA(z)}{|z-\zeta|^{q}}\right)d|\mu|(\zeta).
\end{align*}
For each $\zeta\in K$, the translated set $\Omega-\zeta$ is contained in
$D(0,R)$ by the definition of $R$.  Hence
\[
 \int_\Omega\frac{dA(z)}{|z-\zeta|^{q}}
 \leq\int_{D(0,R)}\frac{dA(u)}{|u|^{q}}
 =\frac{2\pi R^{2-q}}{2-q}.
\]
It follows that
\[
 \int_\Omega|\Cauchy_\mu|^q\,dA
 \leq\|\mu\|_{\TV}^{\,q}\frac{2\pi R^{2-q}}{2-q}.
\]
Taking the $q$th root proves \eqref{eq:Lq-bound}.
\end{proof}

\begin{lemma}\label{lem:L1-Cauchy-continuity}
Let $\mu_n,\mu$ be finite complex measures supported in $K$.  Assume that
$\sup_n\|\mu_n\|_{\TV}<\infty$ and that $\mu_n\to\mu$ weak-$*$ in $\M(K)$.
Then
\begin{equation}\label{eq:L1-Cauchy-convergence}
 \Cauchy_{\mu_n}\longrightarrow\Cauchy_\mu
 \qquad\text{in }L^1(\Omega).
\end{equation}
\end{lemma}

\begin{proof}
Fix $\eps>0$ and choose a continuous cutoff $\chi_\eps:[0,\infty)\to[0,1]$ with
$\chi_\eps(t)=0$ for $t\leq\eps$ and $\chi_\eps(t)=1$ for $t\geq 2\eps$; put
\[
 k_\eps(z,\zeta)=\frac{\chi_\eps(|z-\zeta|)}{z-\zeta},
 \qquad |k_\eps|\leq\tfrac1\eps .
\]
The kernel $k_\eps$ is bounded and continuous on $\overline\Omega\times K$.
Hence, for each fixed $z\in\Omega$, weak-$*$ convergence gives
$\int_Kk_\eps(z,\zeta)\,d\mu_n(\zeta)\to\int_Kk_\eps(z,\zeta)\,d\mu(\zeta)$, and
these quantities are bounded by $M/\eps$ with
$M=\sup_n\|\mu_n\|_{\TV}+\|\mu\|_{\TV}$.  Since $A(\Omega)<\infty$, dominated
convergence gives convergence in $L^1(\Omega)$ of the truncated transforms.  For
the singular remainder, Fubini's theorem and
$\int_{D(\zeta,2\eps)}|z-\zeta|^{-1}dA(z)=4\pi\eps$ give
\[
 \int_\Omega\int_K
 \frac{1-\chi_\eps(|z-\zeta|)}{|z-\zeta|}\,
 d|\mu_n-\mu|(\zeta)\,dA(z)
 \leq
 M\sup_{\zeta\in K}\int_{D(\zeta,2\eps)}\frac{dA(z)}{|z-\zeta|}
 =4\pi M\eps .
\]
Therefore
$\limsup_n\|\Cauchy_{\mu_n}-\Cauchy_\mu\|_{L^1(\Omega)}\leq 4\pi M\eps$ for
every $\eps>0$.
\end{proof}

\section{The algebraic-defect functional}\label{sec:defect}

\begin{definition}[Algebraic defect]
\label{def:algebraic-defect}
For a positive measure $\mu$ supported in $K$ define
\begin{equation}\label{eq:defect-functional}
 \mathfrak D_{P,K}(\mu)
 :=\int_K
 \left|P\bigl(z,\Cauchy_\mu(z)\bigr)\right|^{1/d}\,dA(z),
 \qquad d=\deg_wP,
\end{equation}
where $\Cauchy_\mu$ is any almost-everywhere representative of its
$L^1_{\mathrm{loc}}$ class.
\end{definition}

\begin{remark}\label{rem:exponent}
For $0<p<2/d$, the functional
$\int_K|P(z,\Cauchy_\mu(z))|^p\,dA$ is finite and weak-$*$ continuous by
Lemma~\ref{lem:Lq}; all these functionals have the same zero set.  The choice
$p=1/d$ is used in Section~\ref{sec:quantitative}, since away from
$\Pi_P=p_d^{-1}(0)$ it yields a lower bound of the form
$|P(z,w)|^{1/d}\ge c(|w|-R)^+$.  The integrability range is sharp: for
$\mu=m\delta_a$ with $a\notin\Pi_P$, the local singularity is of order
$|z-a|^{-pd}$.
\end{remark}

\begin{remark}[Independence of the ambient integration domain]
\label{rem:ambient}
For every $\Omega \supset K$ and every $\mu\in\M_f(K)$,
\begin{equation}\label{eq:ambient}
 \mathfrak D_{P,K}(\mu)
 =\int_\Omega
 \left|P\bigl(z,\Cauchy_\mu(z)\bigr)\right|^{1/d}\,dA(z).
\end{equation}
Indeed $z\mapsto P(z,\Cauchy_\mu(z))$ is holomorphic on the connected open set
$\bC\setminus K$ and vanishes near infinity, hence vanishes identically there.
In particular, it vanishes on $\Omega\setminus K$.
\end{remark}

\subsection{Continuity and existence of minimizers}

\begin{proposition}\label{prop:defect-continuity}
For every $M>0$ the functional $\mathfrak D_{P,K}$ is finite and weak-$*$
continuous on
$\mathcal P_M(K)=\{\mu\geq0:\supp\mu\subset K,\ \mu(K)\leq M\}$.  In
particular it is finite and weak-$*$ continuous on $\M_f(K)$.
\end{proposition}

\begin{proof}
Since the coefficients $p_j$ are bounded on $\Omega$, there is $C=C(P,\Omega)$
with
\begin{equation}\label{eq:polynomial-growth}
 |P(z,w)|^{1/d}\leq C(1+|w|),
 \qquad z\in\Omega,\ w\in\bC .
\end{equation}
Finiteness follows from \eqref{eq:polynomial-growth} and
Lemma~\ref{lem:Lq} with $q=1$.

By Lemma~\ref{lem:compact-class}, $\mathcal P_M(K)$ is weak-$*$ compact and
metrizable, so it suffices to prove sequential continuity.  Let
$\mu_n\to\mu$ weak-$*$ in $\mathcal P_M(K)$ and write $C_n=\Cauchy_{\mu_n}$,
$C_\infty=\Cauchy_\mu$.  By Lemma~\ref{lem:L1-Cauchy-continuity},
$C_n\to C_\infty$ in $L^1(\Omega)$; passing to a subsequence we may assume in
addition that $C_n\to C_\infty$ almost everywhere on $\Omega$.  Put
\[
 g_n(z)=\bigl|P(z,C_n(z))\bigr|^{1/d},
 \qquad
 g_\infty(z)=\bigl|P(z,C_\infty(z))\bigr|^{1/d} .
\]
Then $g_n\to g_\infty$ almost everywhere, because $w\mapsto|P(z,w)|^{1/d}$ is
continuous.  Fix $q\in(1,2)$.  By \eqref{eq:polynomial-growth} and
Lemma~\ref{lem:Lq},
\[
 \int_\Omega g_n^{\,q}\,dA
 \leq C^q\int_\Omega(1+|C_n|)^{q}\,dA
 \leq C^q2^{q}\bigl(A(\Omega)+A_q^{\,q}M^{q}\bigr),
\]
uniformly in $n$.  A family bounded in $L^q(\Omega)$ with $q>1$ on a set of
finite measure is uniformly integrable, so Vitali's convergence theorem gives
$g_n\to g_\infty$ in $L^1(\Omega)$ and a fortiori
$\int_Kg_n\,dA\to\int_Kg_\infty\,dA$.  Since every subsequence of
$(\mathfrak D_{P,K}(\mu_n))_n$ has a further subsequence converging to
$\mathfrak D_{P,K}(\mu)$, the whole sequence converges.
\end{proof}

\begin{theorem}[Existence and the zero-defect criterion]
\label{thm:main-defect}
Assume $\M_f(K)\neq\varnothing$.  Then $\mathfrak D_{P,K}$ attains its minimum
on $\M_f(K)$, and the following are equivalent:
\begin{enumerate}[label=\textup{(\roman*)}]
\item $\min_{\mu\in\M_f(K)}\mathfrak D_{P,K}(\mu)=0$;
\item there exists $\mu\in\M_f(K)$ such that
\begin{equation}\label{eq:algebraic-ae}
 P\bigl(z,\Cauchy_\mu(z)\bigr)=0
 \qquad\text{for almost every }z\in\bC .
\end{equation}
\end{enumerate}
If these conditions hold, every minimizer satisfies \eqref{eq:algebraic-ae}.
\end{theorem}

\begin{proof}
Existence follows from Lemma~\ref{lem:compact-class} and
Proposition~\ref{prop:defect-continuity}.  If the minimum is $0$ and $\mu_*$ is
a minimizer, then $P(z,\Cauchy_{\mu_*}(z))=0$ for almost every $z\in K$, while
on $\bC\setminus K$ the holomorphic function
$P(z,\Cauchy_{\mu_*}(z))$ vanishes identically by
Remark~\ref{rem:ambient}.  This proves
\eqref{eq:algebraic-ae}.  Conversely \eqref{eq:algebraic-ae} forces
$\mathfrak D_{P,K}(\mu)=0$, so the minimum vanishes.
\end{proof}

\begin{definition}[Minimum algebraic defect]
\label{def:delta}
Set
\begin{equation}\label{eq:delta}
 \delta_P(f;K)
 :=\min_{\mu\in\M_f(K)}\mathfrak D_{P,K}(\mu).
\end{equation}
\end{definition}

We write
\begin{equation}\label{eq:zero-set}
 \mathcal Z_{P,f}(K)
 :=\{\mu\in\M_f(K):\mathfrak D_{P,K}(\mu)=0\}
\end{equation}
for the compact set of \emph{zero-defect algebraic representatives}; no
condition on their supports is included in this notation.

\begin{remark}[Normalization and monotonicity]\label{rem:monotone}
Multiplying $P$ by a nonzero constant $\lambda$ multiplies $\delta_P(f;K)$ by
$|\lambda|^{1/d}$ and changes neither its vanishing nor the set of minimizers;
thus the statement $\delta_P(f;K)=0$ is intrinsic to the irreducible algebraic
curve and the localization set.  Moreover, if $K\subset K'$ are compact convex
sets then $\M_f(K)\subset\M_f(K')$, and by Remark~\ref{rem:ambient} the two
functionals agree on the smaller class; hence
\[
 K\subset K'\quad\Longrightarrow\quad \delta_P(f;K')\leq\delta_P(f;K).
\]
\end{remark}

\section{Concentration estimates}\label{sec:quantitative}

\begin{lemma}[Reverse polynomial bound]\label{lem:reverse}
Let $L\subset\bC\setminus\Pi_P$ be compact and denote 
\[
 \alpha_L=\min_{z\in L}|p_d(z)|>0,\qquad
 \beta_L=\max_{0\leq j\leq d-1}\ \max_{z\in L}|p_j(z)| ,
\]
\[
 R_L=\max\Bigl\{1,\ \frac{2d\beta_L}{\alpha_L}\Bigr\},
 \qquad
 c_L=\Bigl(\frac{\alpha_L}{2}\Bigr)^{1/d}.
\]
Then for all $z\in L$ and all $w\in\bC$,
\begin{equation}\label{eq:reverse}
 |P(z,w)|^{1/d}\ \geq\ c_L\,(|w|-R_L)^{+} ,
\end{equation}
and every root $w$ of $P(z,\cdot)$ with $z\in L$ satisfies $|w|\leq R_L$.
\end{lemma}

\begin{proof}
Let $z\in L$ and $|w|\geq R_L\geq1$.  Then
\begin{align*}
 |P(z,w)|
 &\geq\alpha_L|w|^d-\beta_L\sum_{j=0}^{d-1}|w|^{j}
 \geq\alpha_L|w|^d-d\beta_L|w|^{d-1}\\
 &=|w|^{d-1}\bigl(\alpha_L|w|-d\beta_L\bigr)
 \geq\frac{\alpha_L}{2}|w|^{d},
\end{align*}
since $|w|\geq 2d\beta_L/\alpha_L$.  Hence $|P(z,w)|^{1/d}\geq c_L|w|\geq
c_L(|w|-R_L)$, and for $|w|<R_L$ the right-hand side of \eqref{eq:reverse}
vanishes.  The same computation shows $P(z,w)\neq0$ when $|w|\geq R_L$.
\end{proof}

\begin{lemma}[Mass from the Cauchy transform]\label{lem:mass}
Let $\mu\geq0$ be a finite measure on $\bC$ and $D(a,2r)\subset\bC$.  Then
\begin{equation}\label{eq:mass-from-C}
 \mu\bigl(D(a,r)\bigr)
 \leq\frac{1}{\pi r}\int_{D(a,2r)}|\Cauchy_\mu|\,dA .
\end{equation}
\end{lemma}

\begin{proof}
Choose $\varphi\in C_c^\infty(D(a,2r))$ with $0\leq\varphi\leq1$, $\varphi=1$ on
$D(a,r)$ and $|\nabla\varphi|\leq2/r$, so that
$|\partial_{\bar z}\varphi|=\frac12|\nabla\varphi|\leq1/r$.  By
\eqref{eq:dbar-Cauchy},
\[
 \pi\mu\bigl(D(a,r)\bigr)\leq\pi\int\varphi\,d\mu
 =\Bigl|\int\Cauchy_\mu\,\partial_{\bar z}\varphi\,dA\Bigr|
 \leq\frac1r\int_{D(a,2r)}|\Cauchy_\mu|\,dA . \qedhere
\]
\end{proof}

\begin{theorem}[Concentration bound]\label{thm:concentration}
Let $L\subset\Omega\setminus\Pi_P$ be compact; by Remark~\ref{rem:ambient} the
neighborhood $\Omega$ may be enlarged without changing $\mathfrak D_{P,K}$, so
this covers every compact $L\subset\bC\setminus\Pi_P$.  Let
\[
 r_0=\tfrac14\dist\bigl(L,\ \partial\Omega\cup\Pi_P\bigr)>0,
 \qquad
 L'=\{z:\dist(z,L)\leq2r_0\}\subset\Omega\setminus\Pi_P .
\]
Write $R=R_{L'}$ and $c=c_{L'}$ for the constants of Lemma~\ref{lem:reverse}.
Then for every $\mu\in\M_f(K)$, every $a\in L$ and every $0<r\leq r_0$,
\begin{equation}\label{eq:concentration}
 \mu\bigl(D(a,r)\bigr)
 \ \leq\ 4R\,r+\frac{1}{\pi c}\cdot\frac{\mathfrak D_{P,K}(\mu)}{r}.
\end{equation}
\end{theorem}

\begin{proof}
Fix $a\in L$ and $0<r\leq r_0$, so that $D(a,2r)\subset L'\subset\Omega$.  By
Lemma~\ref{lem:reverse}, for almost every $z\in D(a,2r)$,
\[
 |\Cauchy_\mu(z)|
 \leq R+\bigl(|\Cauchy_\mu(z)|-R\bigr)^{+}
 \leq R+\frac1c\bigl|P(z,\Cauchy_\mu(z))\bigr|^{1/d} .
\]
Integrating over $D(a,2r)$ and using Remark~\ref{rem:ambient},
\[
 \int_{D(a,2r)}|\Cauchy_\mu|\,dA
 \leq 4\pi R\,r^2+\frac1c\int_\Omega\bigl|P(z,\Cauchy_\mu(z))\bigr|^{1/d}dA
 =4\pi R\,r^2+\frac{\mathfrak D_{P,K}(\mu)}{c}.
\]
Now apply Lemma~\ref{lem:mass}.
\end{proof}

\begin{corollary}[Linear growth at zero defect]\label{cor:linear-growth}
Let $L$, $r_0$, $R$ be as in Theorem~\ref{thm:concentration} and suppose
$\mathfrak D_{P,K}(\mu)=0$ for some $\mu\in\M_f(K)$.  Then
\begin{equation}\label{eq:linear-growth}
 \mu\bigl(D(a,r)\bigr)\leq 4R\,r
 \qquad\text{for all }a\in L,\ 0<r\leq r_0 .
\end{equation}
Consequently $\mu$ has no atom in $L$, and $\mu(E)=0$ for every Borel set
$E\subset L$ with $\Hone(E)=0$.
\end{corollary}

\begin{proof}
Inequality \eqref{eq:linear-growth} is \eqref{eq:concentration} with
$\mathfrak D_{P,K}(\mu)=0$.  An atom of mass $m>0$ at $a\in L$ would give
$m\leq\mu(D(a,r))\leq4Rr$ for all small $r$, hence $m=0$.  For the last
assertion, let $E\subset L$ with $\Hone(E)=0$ and let $\eta>0$.  Cover $E$ by
sets $S_i$ of diameter $t_i<r_0/2$ with $\sum_it_i<\eta$ and $S_i\cap E\neq
\varnothing$; picking $a_i\in S_i\cap E\subset L$ we get $S_i\subset
D(a_i,2t_i)$ with $2t_i<r_0$, and hence
$\mu(E)\leq\sum_i\mu(D(a_i,2t_i))\leq8R\sum_it_i<8R\eta$.  Let $\eta\downarrow0$.
\end{proof}

\begin{corollary}[Small defect forbids concentration]\label{cor:small-defect}
With $L$, $r_0$, $R$, $c$ as in Theorem~\ref{thm:concentration}, set
$\eta=\mathfrak D_{P,K}(\mu)$.  If $0<\eta\leq4\pi Rc\,r_0^2$, then
\[
 \sup_{a\in L}\ \mu\Bigl(D\bigl(a,\sqrt{\eta/(4\pi Rc)}\bigr)\Bigr)
 \ \leq\ 4\sqrt{\frac{R\,\eta}{\pi c}} .
\]
\end{corollary}

\begin{proof}
Minimize the right-hand side of \eqref{eq:concentration} over $r>0$; the minimum
is attained at $r=\sqrt{\eta/(4\pi Rc)}$ and equals $4\sqrt{R\eta/(\pi c)}$.
\end{proof}

\begin{corollary}[A quantitative lower bound for a positive defect]
\label{cor:atom-bound}
With $L$, $r_0$, $R$, and $c$ as in
Theorem~\ref{thm:concentration}, suppose that $\mu\in\M_f(K)$ has an atom of
mass $m$ at a point $a\in L$.  Then
\begin{equation}\label{eq:atom-bound}
 \mathfrak D_{P,K}(\mu)\ge c_{\rm at}m^2,
 \qquad
 c_{\rm at}
 =\min\left\{\frac{\pi c}{16R},\frac{\pi c r_0}{2a_0}\right\}>0.
\end{equation}
In particular, if every $\mu\in\M_f(K)$ has an atom of mass at least $m_0$ at
some point of $L$, then $\delta_P(f;K)\ge c_{\rm at}m_0^2$.
\end{corollary}

\begin{proof}
If $m\le8Rr_0$, take $r=m/(8R)$ in \eqref{eq:concentration}; since
$\mu(D(a,r))\ge m$, this gives
\[
 m\le\frac m2+\frac{8R}{\pi c\,m}\mathfrak D_{P,K}(\mu),
\]
and hence the first constant in \eqref{eq:atom-bound}.  If $m>8Rr_0$, take
$r=r_0$.  Then $4Rr_0<m/2$, so
\[
 \mathfrak D_{P,K}(\mu)
 \ge \pi c r_0\bigl(m-4Rr_0\bigr)
 >\frac{\pi c r_0}{2}m
 \ge\frac{\pi c r_0}{2a_0}m^2,
\]
because $m\le\mu(K)=a_0$.
\end{proof}

The linear growth estimate does not exclude an absolutely continuous part.
The geometric conclusions below use the planar-null hypothesis in the
algebraic-support theorem.

\section{Support under a null-support hypothesis}
\label{sec:support}

\begin{theorem}[Algebraic-support theorem; {\cite{BBB}}]
\label{thm:external-support}
Let $\mu$ be a finite positive compactly supported measure on $\bC$.  Assume
that
\[
 A(\supp\mu)=0
 \qquad\text{and}\qquad
 P\bigl(z,\Cauchy_\mu(z)\bigr)=0
 \quad\text{for almost every }z\in\bC,
\]
where $P\in\bC[z,w]$ is irreducible.  Then $\supp\mu$ is a compact
real-analytic set of real dimension at most one.  Consequently, it is a finite
union of compact semi-analytic arcs and finitely many points; in particular, it
is $1$-rectifiable and has finite
$\Hone$-measure.
\end{theorem}

\begin{remark}[Scope of the quoted theorem]
The paper \cite{BBB} studies subharmonic functions whose Laplacian is supported
on a planar null set and proves the algebraic support description within that
class.  Its probability normalization causes no restriction here: replacing
$\mu$ by $\mu/\mu(\bC)$ preserves the support and merely rescales the second
variable in the defining polynomial.  Thus Theorem~\ref{thm:external-support}
cannot be deduced from the almost-everywhere algebraic equation alone.  The results of
Sections~\ref{sec:moment}--\ref{sec:quantitative} do not use it.  The removal of
the null-support assumption is Question~\ref{q:null-support}.
\end{remark}

\begin{proposition}[Relation to mother-body measures]
\label{prop:motherbody-relation}
The following statements hold.
\begin{enumerate}[label=\textup{(\roman*)}]
\item Every positive mother-body measure of $f$ supported in $K$ belongs to
$\mathcal Z_{P,f}(K)$; in particular, its existence implies
$\delta_P(f;K)=0$.
\item If $\mu\in\mathcal Z_{P,f}(K)$ and $A(\supp\mu)=0$, then
Theorem~\ref{thm:external-support} applies, and $\mu$ is a positive mother-body
measure in the sense of \cite{BoSh}.
\item Consequently, if every measure in $\mathcal Z_{P,f}(K)$ has planar-null
support, then
\begin{equation}\label{eq:conditional-equivalence}
 \delta_P(f;K)=0
 \quad\Longleftrightarrow\quad
 f\text{ has a positive mother-body measure supported in }K.
\end{equation}
\end{enumerate}
\end{proposition}

\begin{proof}
A mother-body measure is supported on finitely many curves and points, hence on
a planar null set.  Off its support, its Cauchy transform is locally a branch
of $P(z,w)=0$; therefore the algebraic equation holds almost everywhere and the
defect vanishes.  Conversely, a zero-defect measure satisfies the algebraic
equation almost everywhere by Theorem~\ref{thm:main-defect}; under the stated
null-support hypothesis, Theorem~\ref{thm:external-support} gives the required
geometric support description.
\end{proof}

\begin{corollary}[Density and length under the null-support hypothesis]\label{cor:length}
Let $\mu\in\mathcal Z_{P,f}(K)$ and assume $A(\supp\mu)=0$.  Put
$\Gamma=\supp\mu$, and let $L$, $r_0$, and $R$ be as in
Theorem~\ref{thm:concentration}.  Then
\begin{equation}\label{eq:density-representation}
 \mu\lfloor L=\theta\,\Hone\lfloor(\Gamma\cap L).
\end{equation}
Here $\mu\lfloor L$ denotes the restriction of $\mu$ to $L$, while
$\Hone\lfloor(\Gamma\cap L)$ denotes one-dimensional Hausdorff measure
restricted to $\Gamma\cap L$.  Thus \eqref{eq:density-representation} means
that, for every Borel set $E\subset\bC$,
\begin{equation}\label{eq:density-representation-expanded}
 \mu(E\cap L)=\int_{E\cap\Gamma\cap L}\theta\,d\Hone.
\end{equation}
Since $\Gamma$ is $1$-rectifiable, $\Hone$ agrees with ordinary arclength on
its rectifiable arcs.  Consequently, the formula says that away from
$\Pi_P$ the measure $\mu$ is given by a density $\theta$ with respect to
arclength on its support.  The estimate $\theta\le2R$ holds for
$\Hone$-almost every point of $\Gamma\cap L$; this means that it may fail only
on a subset of $\Gamma\cap L$ having zero $\Hone$-measure, or equivalently
zero arclength.  If, in addition, $p_d$ has no
zero in $K$, then
\begin{equation}\label{eq:length-bound}
 \Hone(\supp\mu)\ge\frac{a_0}{2R},
\end{equation}
where $R$ is computed on a compact neighborhood of $K$ disjoint from
$\Pi_P$.
\end{corollary}

\begin{proof}
By Theorem~\ref{thm:external-support}, $\Gamma$ is $1$-rectifiable and has
finite $\Hone$-measure.  Corollary~\ref{cor:linear-growth} shows that
$\mu\lfloor L$ charges no $\Hone$-null set.  Since $\mu$ is supported on
$\Gamma$, this is precisely the absolute continuity
\[
 \mu\lfloor L\ll\Hone\lfloor(\Gamma\cap L).
\]
The Radon--Nikodym theorem therefore gives a Borel function $\theta\ge0$ for
which \eqref{eq:density-representation}--\eqref{eq:density-representation-expanded}
hold.  At $\Hone$-almost every $a\in\Gamma\cap L$, the point $a$ is both a
Lebesgue point of $\theta$ with respect to
$\Hone\lfloor(\Gamma\cap L)$ and a density point of $\Gamma\cap L$ relative
to $\Gamma$.  The differentiation theorem and the density theorem for
rectifiable sets~\cite[Ch.~16]{Mattila} therefore give
\[
 \theta(a)
 =\lim_{r\downarrow0}
   \frac{\mu(L\cap D(a,r))}
        {\Hone(\Gamma\cap L\cap D(a,r))},
 \qquad
 \Hone(\Gamma\cap L\cap D(a,r))=2r(1+o(1)).
\]
The second limit says that a small disk centered at a typical point of the
rectifiable support contains asymptotically a curve segment of length $2r$.
Since $\mu(L\cap D(a,r))\le\mu(D(a,r))\le4Rr$, division by
$2r(1+o(1))$ yields $\theta(a)\le2R$.  If
$p_d$ has no zero in $K$, choose $L$ to contain $K$ and be disjoint from
$\Pi_P$.  Since $\mu(K)=a_0$,
\[
 a_0=\int_\Gamma\theta\,d\Hone
 \le2R\,\Hone(\Gamma),
\]
which proves \eqref{eq:length-bound}.
\end{proof}

\section{The rational case}\label{sec:rational}

Assume $d=1$ and
\[
 P(z,w)=Q(z)w-R(z),\qquad \gcd(Q,R)=1,\qquad f=\frac RQ ,
\]
so that irreducibility of $P$ is equivalent to $\gcd(Q,R)=1$ and $\Pi_P=Q^{-1}(0)$
is the pole set of $f$.

\begin{lemma}[The defect as a Cauchy transform]\label{lem:rational-identity}
For $\mu\in\M_f(K)$ let $Q\mu$ denote the complex measure $Q(z)\,d\mu(z)$.  Then
\[
 Q\Cauchy_\mu-R=\Cauchy_{Q\mu}
 \qquad\text{in }L^1_{\mathrm{loc}}(\bC),
\]
this function vanishes identically on $\bC\setminus K$, and
\begin{equation}\label{eq:rational-identity}
 \mathfrak D_{P,K}(\mu)
 =\int_K\bigl|Q(z)\Cauchy_\mu(z)-R(z)\bigr|\,dA(z)
 =\bigl\|\Cauchy_{Q\mu}\bigr\|_{L^1(\bC)} .
\end{equation}
In particular $\mathfrak D_{P,K}$ is a convex functional on $\M_f(K)$, and all
holomorphic moments of $Q\mu$ vanish:
\begin{equation}\label{eq:vanishing-moments}
 \int_Kz^{n}Q(z)\,d\mu(z)=0\qquad\text{for all }n\geq0 .
\end{equation}
\end{lemma}

\begin{proof}
Both $Q\Cauchy_\mu-R$ and $\Cauchy_{Q\mu}$ lie in $L^1_{\mathrm{loc}}$ and, since
$Q$ and $R$ are holomorphic, both have distributional $\bar\partial$-derivative
$\pi Q\mu$; their difference is therefore entire.  Off $K$ we have
$\Cauchy_\mu=f=R/Q$, so $Q\Cauchy_\mu-R=0$ there, while $\Cauchy_{Q\mu}\to0$ at
infinity; hence the entire difference vanishes and both functions are supported
in $K$.  Formula \eqref{eq:rational-identity} follows, and convexity is clear
because $\mu\mapsto\Cauchy_{Q\mu}$ is affine and the $L^1$ norm is convex.
Finally, the Laurent coefficients of $\Cauchy_{Q\mu}$ at infinity are the
moments $m_n(Q\mu)$, and $\Cauchy_{Q\mu}\equiv0$ near infinity.
\end{proof}

\begin{theorem}[Exact dual formula for the rational defect]
\label{thm:rational-dual}
Let
\[
 \mathcal Y_{Q,f}(K)
 =\{\Cauchy_{Q\mu}:\mu\in\M_f(K)\}\subset L^1(K,dA).
\]
Then $\mathcal Y_{Q,f}(K)$ is compact and convex and
\begin{equation}\label{eq:rational-dual}
 \delta_P(f;K)
 =\operatorname{dist}_{L^1}(0,\mathcal Y_{Q,f}(K))
 =\sup_{\substack{h\in L^\infty(K)\\ \|h\|_\infty\leq1}}
   \ \inf_{\mu\in\M_f(K)}
   \Re\int_K h(z)\Cauchy_{Q\mu}(z)\,dA(z).
\end{equation}
Consequently $\delta_P(f;K)>0$ if and only if there are
$h\in L^\infty(K)$ and $c>0$ such that
\[
 \Re\int_K h\,\Cauchy_{Q\mu}\,dA\geq c
 \qquad\text{for every }\mu\in\M_f(K).
\]
Moreover, for every $\varphi\in C_c^1(\bC)$ with
$\|\partial_{\bar z}\varphi\|_\infty>0$,
\begin{equation}\label{eq:smooth-dual-lower-bound}
 \delta_P(f;K)
 \geq
 \frac{\pi}{\|\partial_{\bar z}\varphi\|_\infty}
 \inf_{\mu\in\M_f(K)}
 \left|\int_K \varphi(z)Q(z)\,d\mu(z)\right|.
\end{equation}
\end{theorem}

\begin{proof}
By Lemmas~\ref{lem:compact-class} and~\ref{lem:L1-Cauchy-continuity}, the map
$\mu\mapsto\Cauchy_{Q\mu}$ is continuous from the weak-$*$ compact convex set
$\M_f(K)$ into $L^1(K)$; hence its image $\mathcal Y_{Q,f}(K)$ is compact and
convex.  Lemma~\ref{lem:rational-identity} identifies the defect with the
$L^1$ norm on this image.  The second equality in
\eqref{eq:rational-dual} is the Hahn--Banach distance formula for a closed
convex subset of the real Banach space underlying $L^1(K)$, whose dual is
$L^\infty(K)$.  The separation statement is the corresponding strict
separation criterion.

For the last assertion put $\nu=Q\mu$.  Distributionally,
$\partial_{\bar z}\Cauchy_\nu=\pi\nu$, and therefore
\[
 \pi\left|\int\varphi\,d\nu\right|
 =\left|\int_{\bC}\Cauchy_\nu\,
        \partial_{\bar z}\varphi\,dA\right|
 \leq \|\partial_{\bar z}\varphi\|_\infty
       \|\Cauchy_\nu\|_{L^1(\bC)}.
\]
Now use Lemma~\ref{lem:rational-identity}, take the infimum over $\mu$, and
obtain \eqref{eq:smooth-dual-lower-bound}.
\end{proof}

\begin{theorem}[Rational case]\label{thm:rational}
With the notation above, $\delta_P(f;K)=0$ if and only if all poles of $f$ are
simple, lie in $K$, and have positive residues.  In that case the unique zero
minimizer is $\mu=\sum_k c_k\delta_{z_k}$, where $z_k$ are the poles and
$c_k=\operatorname{res}_{z_k}f>0$.
\end{theorem}

\begin{proof}
By \eqref{eq:rational-identity}, $\mathfrak D_{P,K}(\mu)=0$ if and only if
$\Cauchy_{Q\mu}\equiv0$, i.e.\ if and only if $Q\mu=0$ by
\eqref{eq:dbar-Cauchy}; equivalently $\mu(\{Q\neq0\})=0$, i.e.\ $\supp\mu$ is
contained in the finite set $Q^{-1}(0)$.  Thus $\mu=\sum_kc_k\delta_{z_k}$ with
$c_k>0$ and $Q(z_k)=0$, $z_k\in K$, and the constraint $\Cauchy_\mu=f$ off $K$
reads $\sum_kc_k/(z-z_k)=R/Q$: all poles of $f$ are simple with positive
residues $c_k$ and lie in $K$.  Conversely, if $f$ has only simple poles $z_k\in
K$ with residues $c_k>0$, then (using $f(z)=O(1/z)$ at infinity) the partial
fraction expansion is $f=\sum_kc_k/(z-z_k)$, and $\mu=\sum_kc_k\delta_{z_k}$
lies in $\M_f(K)$ with $Q\mu=0$.  Uniqueness of the zero minimizer is clear from
the same description.
\end{proof}

\begin{example}[A positive algebraic germ with positive defect]
\label{ex:positive-defect}
Let $\sigma$ be the measure on the unit circle with density
\[
 d\sigma\bigl(e^{i\theta}\bigr)
 =\frac{1}{2\pi}\Bigl(1+\tfrac12\cos\theta\Bigr)\,d\theta\ >\ 0 .
\]
Its holomorphic moments are $m_0=1$, $m_1=\tfrac14$ and $m_n=0$ for $n\geq2$,
so the exterior Cauchy transform is the rational germ
\[
 f(z)=\frac1z+\frac{1}{4z^{2}}=\frac{4z+1}{4z^{2}} ,
 \qquad
 P(z,w)=4z^{2}w-4z-1 ,
\]
which is irreducible with $d=1$, $Q(z)=4z^2$, $R(z)=4z+1$.  Here
$K=\Conv(\supp\sigma)=\overline{\Db}$ is the closed unit disc, and
$\M_f(K)\ni\sigma$ is nonempty.  By Theorem~\ref{thm:rational} a vanishing
defect would force $\supp\mu\subset Q^{-1}(0)=\{0\}$, hence $\mu=a_0\delta_0=
\delta_0$; but then $m_1(\mu)=0\neq\tfrac14$, contradicting the moment
constraint.  Therefore
\[
 \delta_P\bigl(f;\overline{\Db}\bigr)>0 ,
\]
and this positive value is attained.  By Theorem~\ref{thm:rational}, the same
conclusion holds for every compact convex $K'\supset\overline{\Db}$.
\end{example}

\section{Constructing classes of positive algebraic Cauchy transforms}
\label{sec:constructive}

We give a real-line criterion and several constructions of positive algebraic
Cauchy transforms.

\subsection{Algebraic Herglotz functions}

\begin{theorem}[Real-line criterion]\label{thm:herglotz}
Let $E\subset\bR$ be a finite union of compact intervals and points, and let
$f$ be an algebraic function, single-valued and holomorphic on $\bC\setminus E$.
Assume
\begin{equation}\label{eq:herglotz-hyp}
 f(\bar z)=\overline{f(z)},\qquad
 \Im f(z)<0\quad(\Im z>0),\qquad
 f(z)=\frac{a_0}{z}+O(z^{-2}),\quad a_0>0.
\end{equation}
Then $f=\Cauchy_\mu$ on $\bC\setminus E$ for a unique positive measure $\mu$
of mass $a_0$.  The support of $\mu$ is a finite union of compact intervals and
points.  Hence $\mu$ is a positive mother-body measure and, if $P$ is the irreducible
defining polynomial of $f$, then $\delta_P(f;K)=0$ for every compact convex
$K\supset E$.

Conversely, the algebraic Cauchy transform of a nonzero positive measure on
$E$ satisfies \eqref{eq:herglotz-hyp}.
\end{theorem}

\begin{proof}
The function $-f$ is a Herglotz function.  The Herglotz representation theorem
\cite[Chapter~III]{Donoghue}, together with the expansion at infinity, gives a
finite positive measure $\mu$ on $\bR$ such that
\[
 f(z)=\int_{\bR}\frac{d\mu(x)}{z-x},\qquad \mu(\bR)=a_0.
\]
Since $f$ extends holomorphically and is real on every interval contained in
$\bR\setminus E$, Stieltjes inversion gives $\supp\mu\subset E$.

It remains to determine the structure of the support.  Let $S\subset E$ consist
of the endpoints of the interval components together with the real branch
points and poles of $f$.  This set is finite.  On each component $I$ of
$E\setminus S$, the boundary values $f_\pm$ are analytic and
\begin{equation}\label{eq:herglotz-density}
 d\mu(x)=\frac{f_-(x)-f_+(x)}{2\pi i}\,dx
        =-\frac1\pi\Im f_+(x)\,dx.
\end{equation}
The difference $f_--f_+$ is algebraic.  After passing to the normalization of
the corresponding algebraic curve, it is meromorphic; hence it is either
identically zero or has only finitely many zeros.  Since the density in
\eqref{eq:herglotz-density} is nonnegative, the support in $I$ is therefore a
finite union of intervals.  The remaining part of the measure is supported on
$S$.  An atom at $x_0$ is equivalent to a simple pole of $f$, and its mass is
$\operatorname*{res}_{z=x_0}f(z)\ge0$.  Uniqueness follows from the Herglotz
representation.

Conversely, for $z=x+iy$, $y>0$,
\[
 \Im\Cauchy_\mu(z)=-y\int_E\frac{d\mu(t)}{|z-t|^2}<0,
\]
and the symmetry and expansion at infinity are immediate.
\end{proof}

\subsection{Fuss--Catalan and Raney families}

We use the notions of balanced and excellent balanced polynomials from
Section~\ref{sec:BoSh-background}.

\begin{theorem}[Fuss--Catalan family]\label{thm:fuss-catalan}
For every integer $m\ge2$, the polynomial
\begin{equation}\label{eq:fuss-polynomial}
 P_m(z,w)=z^{m-1}w^m-zw+1
\end{equation}
is irreducible and balanced but not excellent balanced.  Its branch
$f_m(z)=z^{-1}+O(z^{-2})$ is the Cauchy transform of the Fuss--Catalan
probability measure supported on
\begin{equation}\label{eq:fuss-support}
 [0,K_m],\qquad K_m=\frac{m^m}{(m-1)^{m-1}}.
\end{equation}
Consequently $f_m$ has a positive mother body and
$\delta_{P_m}(f_m;K)=0$ for every compact convex $K\supset[0,K_m]$.
\end{theorem}

\begin{proof}
Set $W=zf_m$.  Equation \eqref{eq:fuss-polynomial} becomes
\[
 W^m-zW+z=0,
 \qquad W=1+z^{-1}W^m.
\]
The branch $W=1+O(z^{-1})$ is unique, and Lagrange inversion gives
\[
 W(z)=\sum_{k\ge0}\frac{1}{(m-1)k+1}\binom{mk}{k}z^{-k}.
\]
These are the moments of the Fuss--Catalan probability measure on
\eqref{eq:fuss-support}; see \cite[(2.1), (2.7), and (2.10)]{ForresterLiu}.
Hence $f_m=\Cauchy_{\mu_m}$.

For irreducibility, write $F(z,u)=u^m-zu+z$.  As a polynomial of degree one in
$z$, it is irreducible in $\bC[u,z]$, since $u^m$ and $1-u$ are coprime.  The
substitution $u=zw$ sends $zP_m(z,w)$ to $F(z,u)$, so $P_m$ is irreducible over
$\bC(z)$ and therefore in $\bC[z,w]$.  Its exponent pairs are
$(m-1,m),(1,1),(0,0)$, whence $m(P_m)=0$.  The diagonal monomial $zw$ does not
dominate $z^{m-1}w^m$ coordinatewise, so $P_m$ is not excellent balanced.
\end{proof}

\begin{proposition}[Raney family]\label{prop:raney}
Let $p\ge2$ and $1\le r\le p$ be integers.  There exists a positive probability
measure $\mu_{p,r}$ on
$[0,K_p]$, $K_p=p^p(p-1)^{1-p}$, whose Cauchy transform $G_{p,r}$ is algebraic.
If $u=zG_{p,1}(z)$, then
\begin{equation}\label{eq:raney-relations}
 u^p-zu+z=0,\qquad u^r=zG_{p,r}(z).
\end{equation}
Thus $G_{p,r}$ is annihilated by
\begin{equation}\label{eq:raney-resultant}
 \mathcal R_{p,r}(z,w)=
 \operatorname{Res}_u\bigl(u^p-zu+z,\ u^r-zw\bigr),
\end{equation}
and the irreducible factor containing the probability branch admits the positive
mother body $\mu_{p,r}$.
\end{proposition}

\begin{proof}
The moments of $\mu_{p,r}$ are the Raney numbers
\[
 R_{p,r}(k)=\frac{r}{pk+r}\binom{pk+r}{k}.
\]
The support, positivity of the density, and identities
\eqref{eq:raney-relations} are proved in
\cite[(2.1), (2.6), and Proposition~2.1]{ForresterLiu}.  Elimination of $u$
gives \eqref{eq:raney-resultant}.
\end{proof}

Theorem~\ref{thm:fuss-catalan} provides an infinite family covered by the
balanced-polynomial conjecture of~\cite{BoSh} but not by
Theorem~\ref{thm:BoSh-excellent}.

\subsection{Positive sums}

\begin{proposition}\label{prop:positive-sums}
Let $f_j=\Cauchy_{\mu_j}$, $j=1,2$, be algebraic Cauchy transforms of positive
mother-body measures, and let $P_j$ be irreducible defining polynomials of
$w$-degrees $d_j$.  For $\alpha,\beta>0$,
\[
 f=\alpha f_1+\beta f_2
\]
is algebraic and is represented by the positive mother-body measure
$\alpha\mu_1+\beta\mu_2$.  An annihilating polynomial is
\begin{equation}\label{eq:sum-resultant}
 \operatorname{Res}_u\left(
 \alpha^{d_1}P_1\left(z,\frac{u}{\alpha}\right),
 \beta^{d_2}P_2\left(z,\frac{w-u}{\beta}\right)
 \right).
\end{equation}
\end{proposition}

\begin{proof}
The equality
$\Cauchy_{\alpha\mu_1+\beta\mu_2}=\alpha f_1+\beta f_2$ holds off the union of
the supports.  Formula \eqref{eq:sum-resultant} follows by eliminating
$u=\alpha f_1$; its irreducible factor containing $f$ is a defining polynomial.
The support remains a finite union of analytic arcs and points.
\end{proof}

\begin{example}[A quartic supported on two intervals]\label{ex:two-arcsine}
Let $a<b<c<d$ be real and put
\[
 A(z)=(z-a)(z-b),\qquad B(z)=(z-c)(z-d).
\]
For $\alpha,\beta>0$, the measure
$\alpha\omega_{[a,b]}+\beta\omega_{[c,d]}$, where $\omega_I$ denotes normalized
arcsine measure on $I$, has Cauchy transform
\[
 f(z)=\frac{\alpha}{\sqrt{A(z)}}+\frac{\beta}{\sqrt{B(z)}}
\]
and satisfies
\begin{equation}\label{eq:two-arcsine-quartic}
 \left( A(z)B(z)w^2-\alpha^2B(z)-\beta^2A(z)\right)^2
 -4\alpha^2\beta^2A(z)B(z)=0.
\end{equation}
The polynomial is irreducible.  Indeed,
$\bC(z)(\sqrt A,\sqrt B)/\bC(z)$ is biquadratic and the four conjugates
$\pm\alpha/\sqrt A\pm\beta/\sqrt B$ are distinct.  The primitive quartic in
\eqref{eq:two-arcsine-quartic} is therefore the minimal polynomial of $f$.
\end{example}

\subsection{Polynomial pushforwards}

\begin{proposition}\label{prop:pushforward}
Let $\mu$ be a positive mother-body measure whose Cauchy transform is algebraic,
and let $\phi\in\bC[t]$ be nonconstant.  Then $\nu=\phi_*\mu$ is a positive
mother-body measure and $\Cauchy_\nu$ is algebraic.  If $z$ is a regular value of
$\phi$ outside $\supp\nu$, then
\begin{equation}\label{eq:pushforward-trace}
 \Cauchy_\nu(z)=
 \sum_{\phi(\xi)=z}\frac{\Cauchy_\mu(\xi)}{\phi'(\xi)}.
\end{equation}
Near infinity the right-hand side is the field trace of
$\Cauchy_\mu(\xi)/\phi'(\xi)$ over $\bC(z)$.
\end{proposition}

\begin{proof}
For a regular value $z$,
\[
 \frac1{z-\phi(t)}=
 \sum_{\phi(\xi)=z}\frac1{\phi'(\xi)(\xi-t)}.
\]
Integration gives \eqref{eq:pushforward-trace}.  For $|z|$ large all roots $\xi$ lie in the exterior component of
$\bC\setminus\supp\mu$, and the sum is a field trace in a finite algebraic
extension of $\bC(z)$.  On every component of $\bC\setminus\supp\nu$, formula
\eqref{eq:pushforward-trace} is a sum of local algebraic branches.  These sums
belong to a finite algebraic extension of $\bC(z)$ and are annihilated by a
common polynomial.  Thus $\Cauchy_\nu$ is algebraic almost everywhere.  A compact
one-dimensional real-analytic image admits a finite real-analytic
stratification into arcs and points.  Hence $\phi_*\mu$ has mother-body
support.
\end{proof}

\begin{example}[A quartic supported on a parabolic arc]
\label{ex:parabolic-pushforward}
Let
\[
 d\omega(t)=\frac{dt}{\pi\sqrt{1-t^2}},\qquad -1<t<1,
\]
and $\phi(t)=t^2+it$.  The measure $\nu=\phi_*\omega$ is supported on
$\{t^2+it:-1\le t\le1\}$.  Put $D(z)=z^2-2z+2$.  Its Cauchy transform $G$ is the
probability branch of
\begin{equation}\label{eq:parabolic-quartic}
 D(z)^2(4z-1)G^4-2D(z)(2z-3)G^2-1=0.
\end{equation}

Indeed, if $s^2=4z-1$, $\xi_\pm=(-i\pm s)/2$, and
$g(\xi)=(\xi^2-1)^{-1/2}$ is normalized at infinity, then
$G=(g(\xi_+)-g(\xi_-))/s$.  With $u_\pm=\xi_\pm^2-1$ one has
\[
 u_++u_-=2z-3,
 \qquad u_+u_-=D(z),
\]
which gives \eqref{eq:parabolic-quartic}.  As a quadratic in $G^2$, its
discriminant is $16D(z)^3$.  Over $\bC(z)(\sqrt D)$ the norm of $G^2$ is
$-1/(D(z)^2(4z-1))$, which is not a square in $\bC(z)$.  Thus the degree of $G$
over $\bC(z)$ is four, and \eqref{eq:parabolic-quartic} is irreducible.
\end{example}

\subsection{Branch graphs}

\begin{theorem}[Positive branch-graph criterion]\label{thm:branch-graph}
Let $\Gamma$ be a finite compact embedded real-analytic graph, let
$U_0,\ldots,U_N$ be the components of $\bC\setminus\Gamma$, with $U_0$
unbounded, and let $S\subset\bC\setminus\Gamma$ be the finite set of poles of
the piecewise function below.  On $U_j$ choose a single-valued meromorphic
branch $F_j$ of an irreducible equation
$P(z,w)=0$.  Assume:
\begin{enumerate}[label=\textup{(\roman*)}]
\item $F_0(z)=a_0/z+O(z^{-2})$, $a_0>0$;
\item the poles are simple, belong to $S$, and have nonnegative real residues;
\item the branches have continuous boundary values on open edges and are
locally bounded at the vertices;
\item every edge can be oriented so that
\begin{equation}\label{eq:positive-edge-jump}
 d\mu_e(z)=\frac{F_-(z)-F_+(z)}{2\pi i}\,dz
\end{equation}
is nonnegative, where $F_+$ and $F_-$ are the boundary values from the left and
right.
\end{enumerate}
Then
\[
 \mu=\sum_e\mu_e+
 \sum_{a\in S}\operatorname*{res}_{z=a}F(z)\,\delta_a
\]
is a positive mother-body measure of mass $a_0$, and the piecewise function
$F|_{U_j}=F_j$ equals $\Cauchy_\mu$ almost everywhere.  Conversely, every
positive mother body with this graph structure and boundary regularity gives
such data.
\end{theorem}

\begin{proof}
The edge measures are finite by the boundary regularity.  The distributional
jump formula, together with the pole contributions, gives
$\partial_{\bar z}F=\pi\mu$.  Local boundedness at the vertices excludes
additional point masses.  Thus $F-\Cauchy_\mu$ is distributionally holomorphic
and belongs to $L^1_{\rm loc}(\bC)$; it is therefore entire.  It tends to zero at
infinity and hence vanishes.  Comparison of the coefficients of $z^{-1}$ gives
$\mu(\bC)=a_0$.  The converse is the Plemelj formula.
\end{proof}

For a cyclic equation $A(z)w^m+D(z)=0$, write $h^m=-D/A$ and
$F_k=\zeta_m^kh$.  Condition \eqref{eq:positive-edge-jump} requires
$h(z)\,dz$ to have a prescribed constant argument on every edge.  Thus the
edges are trajectories, with phases determined by adjacent labels, of the
meromorphic $m$-differential
\[
 -\frac{D(z)}{A(z)}(dz)^m.
\]
In addition one needs a compatible $\mathbb Z_m$-labelling of the complementary
components and nonnegative residues.  For $m=2$ this is the quadratic
trajectory condition of~\cite{BoSh}.

\section{Quadratic equations}\label{sec:quadratic}

Let
\[
 P(z,w)=A(z)w^2+B(z)w+D(z),
 \qquad \Delta=B^2-4AD,
\]
and let $F_\pm$ be the two branches.  Along a smooth arc on which the branch
changes, the Plemelj formula gives
\[
 d\mu(z)=\frac{F_-(z)-F_+(z)}{2\pi i}\,dz.
\]
Hence every such arc is a horizontal trajectory of
\begin{equation}\label{eq:branch-difference-QD}
 \Theta=-\bigl(F_+-F_-\bigr)^2dz^2
       =-\frac{\Delta(z)}{A(z)^2}\,dz^2;
\end{equation}
see \cite[Proposition~9]{BoSh}.

The following is the positive version of \cite[Theorem~12]{BoSh}.

\begin{theorem}\label{thm:quadratic-finite}
Assume
\[
 \deg A=n+2,\qquad \deg B\le n+1,\qquad \deg D\le n,
\]
that $A$ and $B$ are coprime, $\deg\Delta=2n+2$, and the positive branch at
infinity is $a_0/z+O(z^{-2})$.  Suppose that the graph of finite critical
trajectories of \eqref{eq:branch-difference-QD} contains every zero of $\Delta$
and that $\Theta$ has no closed horizontal trajectories.  Then the positive
mother-body measures are obtained from the finitely many spanning
multisubgraphs without isolated vertices for which the associated section has
nonnegative edge jumps and only simple poles with nonnegative residues.

Every such measure supported in $K$ has zero defect.  Conversely, every
zero-defect measure in $\M_f(K)$ with planar-null support occurs in this list.
\end{theorem}

\begin{proof}
The classification of real mother-body measures by spanning multisubgraphs is
\cite[Theorem~12]{BoSh}.  Positivity is equivalent to nonnegativity of the jump
measures and pole residues.  Proposition~\ref{prop:motherbody-relation} gives
the assertions concerning zero defect.
\end{proof}

For vanishing middle coefficient there is a topological criterion.  Assume
\[
 A(z)w^2+D(z)=0,
\]
where $A,D$ are coprime, $A$ is monic of degree $n+2$, and $D$ has degree $n$
and negative leading coefficient.  Put
$\Psi=D(z)A(z)^{-1}dz^2$.  If $\Psi$ is Strebel, let $K_\Psi$ be its critical
graph and let $F_\Psi$ be the forest obtained by deleting all simple cycles.

\begin{theorem}\label{thm:strebel-zero}
Assume $F_\Psi\subset K$.  The following are equivalent:
\begin{enumerate}[label=\textup{(\roman*)}]
\item no edge of $K_\Psi$ is attached to a simple cycle from its interior;
\item the positive branch admits a positive mother-body measure supported on
$F_\Psi$.
\end{enumerate}
If these conditions hold, then $\delta_P(f;K)=0$.  Conversely, if a zero-defect
measure in $\M_f(K)$ has planar-null support, then \textup{(i)} and
\textup{(ii)} hold.
\end{theorem}

\begin{proof}
The equivalence and the description of the support follow from
\cite[Theorem~13 and Proposition~14]{BoSh}.  The defect assertions follow from
Proposition~\ref{prop:motherbody-relation}.
\end{proof}

\section{Stability}\label{sec:stability}

Let $P_n(z,w)=\sum_{j=0}^dp_{j,n}(z)w^j$ and
$P(z,w)=\sum_{j=0}^dp_j(z)w^j$ have the same $w$-degree $d$, and assume
$p_{j,n}\to p_j$ uniformly on a neighborhood of $K$.  Let $f_n,f$ be positive
algebraic germs satisfying $P_n(z,f_n(z))=0$ and $P(z,f(z))=0$, and assume that
$\M_{f_n}(K)$ and $\M_f(K)$ are nonempty.

\begin{theorem}\label{thm:stability}
Suppose $m_k(f_n)\to m_k(f)$ for every $k\ge0$.  Then
\[
 \liminf_{n\to\infty}\delta_{P_n}(f_n;K)\ge\delta_P(f;K).
\]
If a minimizer $\mu$ for $\delta_P(f;K)$ admits a recovery sequence
$\mu_n\in\M_{f_n}(K)$ with $\mu_n\to\mu$ weak-$*$, then
\[
 \delta_{P_n}(f_n;K)\longrightarrow\delta_P(f;K).
\]
\end{theorem}

\begin{proof}
Set
$\varepsilon_n=\max_j\sup_K|p_{j,n}-p_j|$.  Since
$||u|^{1/d}-|v|^{1/d}|\le|u-v|^{1/d}$,
\[
 \sup_{\substack{\nu\ge0,\ \supp\nu\subset K\\ \nu(K)\le M}}
 |\mathfrak D_{P_n,K}(\nu)-\mathfrak D_{P,K}(\nu)|
 \le C\varepsilon_n^{1/d}
 \left(A(K)+
 \sup_{\substack{\nu\ge0,\ \supp\nu\subset K\\ \nu(K)\le M}}
 \|\Cauchy_\nu\|_{L^1(K)}\right),
\]
and the right-hand side tends to zero by Lemma~\ref{lem:Lq}.  Choose minimizers
$\nu_n\in\M_{f_n}(K)$.  Their masses are bounded, so a subsequence realizing
the lower limit converges weak-$*$ to a positive measure $\nu$ on $K$.  The
moment convergence gives $\nu\in\M_f(K)$, and
Proposition~\ref{prop:defect-continuity} yields
\[
 \liminf_n\delta_{P_n}(f_n;K)
 =\mathfrak D_{P,K}(\nu)\ge\delta_P(f;K).
\]
For a recovery sequence of a minimizer, the same uniform estimate and weak-$*$
continuity give the reverse upper bound.
\end{proof}

\section{Open problems}\label{sec:open}

\begin{question}[Null support]\label{q:null-support}
Let $\mu$ be a positive compactly supported measure and suppose that
$P(z,\Cauchy_\mu(z))=0$ almost everywhere for an irreducible
$P\in\bC[z,w]$.  Must $\supp\mu$ have planar measure zero?  Must it be a finite
union of compact analytic arcs and isolated points?
\end{question}

An affirmative answer would identify zero defect with the existence of a
positive mother body without an additional support hypothesis.

\begin{question}[Effective existence]
Give effective conditions on $P(z,w)$ and a compact convex set $K$ ensuring
that its positive branch has a mother-body measure supported in $K$.  In higher
degree, determine when the branches admit a finite positive graph as in
Theorem~\ref{thm:branch-graph}; for cyclic equations this includes finiteness of
the relevant $m$-differential trajectory graph and compatibility of the face
labels.
\end{question}

\begin{question}[Quantitative defect]
Find computable lower bounds for $\delta_P(f;K)$ arising from non-atomic
obstructions.  In the rational case, can the dual formula
\eqref{eq:rational-dual} be evaluated or bounded directly from multiple poles
or residues of the wrong sign?
\end{question}

\medskip\noindent
\emph{Acknowledgements.} The author wants to thank his friend and coauthor Rikard B\o gvad for his interest in this project. Some results have been obtained using GPT-5.6 Sol which has been provided with an initial draft of the manuscript. It also helped with restructuring, shortening and polishing of the original text.

\end{document}